\documentclass[a4paper,12pt]{amsart}
\usepackage{amsmath,amssymb,amsthm}
\usepackage[left=2.5cm,right=2.5cm]{geometry}
\usepackage{enumerate}
\usepackage{mathtools}
\usepackage[hidelinks]{hyperref}
\usepackage{verbatim}
\usepackage{hyperref,xcolor}
\hypersetup{colorlinks=true, citecolor=red, linkcolor=green}
\makeatletter
\@namedef{subjclassname@2020}{%
	\textup{2020} Mathematics Subject Classification}
\makeatother

\newtheorem{thm}{Theorem}[section]

\newtheorem{prop}[thm]{Proposition}
\newtheorem{lem}[thm]{Lemma}

\theoremstyle{definition}
\newtheorem{defn}[thm]{Definition}

\newtheorem{rem}[thm]{Remark}

\def\D{\mathbb D}

\begin{document}
	
	\title[critical cyclicity in Dirichlet-type spaces on the bidisk]
	{critical cyclicity in Dirichlet-type spaces on the bidisk}

    \author[Beslikas]{Athanasios Beslikas}
\address{A. Beslikas \\Doctoral School of Exact and Natural Studies,
Institute of Mathematics,
Faculty of Mathematics and Computer Science,
Jagiellonian University,
\L{}ojasiewicza 6, PL30348, Cracow, Poland}
\email{athanasios.beslikas@doctoral.uj.edu.pl}
	
	\author[Torkinejad Ziarati]{Pouriya Torkinejad Ziarati}
	
	\address{P.T.Ziarati\\
		Institut de Math\'ematiques de Toulouse; UMR5219 \\
		Universit\'e de Toulouse; CNRS \\
		UPS, F-31062 Toulouse Cedex 9, France} 
        \email{pouriya.torkinejad\_ziarati@math.univ-toulouse.fr}

	\begin{abstract} Consider the Dirichlet-type spaces on the bidisc defined by 
    $$\mathcal{D}_{\beta}(\mathbb{D}^2)=\Bigg\{f(z_1,z_2)=\sum_{k,l}a_{kl}z_1^kz_2^l\in\mathcal{O}(\mathbb{D}^2): \sum_{k,l\ge 0}|a_{kl}|^2(k+l+1)^{\beta}<+\infty\Bigg\}.$$
    Given $\beta_c\in(0,2],$ we construct a function $f$ that belongs to the Dirichlet-type space $\mathcal D_{2}(\mathbb{D}^2)$ of the bidisk and is cyclic in $\mathcal D_\beta(\mathbb{D}^2)$ if and only if $\beta\leq \beta_{c}.$ We also show that the critical index satisfies $\beta_c=2-\mathrm{dim}_H(\mathcal{Z}(f)\cap \mathbb{T}^2),$ where $\mathrm{dim_H}(\mathcal{Z}(f)\cap \mathbb{T}^2)$ is the Hausdorff dimension of the zero set of the function $f$ on the two-torus $\mathbb{T}^2.$   
              
        \end{abstract}

	
	\subjclass[2020]{32F45}
	
	\keywords{Dirichlet-type Spaces, Weighted Bergman Spaces, Cyclic Vectors, Shift Operator}
	
	\maketitle
\section{Introduction and main result}  
\subsection{Cyclicity in Banach spaces}
\noindent Let $\mathcal A$ be a Banach space of holomorphic functions on a
    domain $\Omega\subset\mathbb C^d$. The multiplier space of
    $\mathcal A$, denoted by $\mathcal M(\mathcal A)$, consists of the
    functions $\varphi$ such that $\varphi g\in\mathcal A$ for every
    $g\in\mathcal A$. Assume that polynomials are multipliers and are
    dense in $\mathcal A$. For $f\in\mathcal A$, let
    \[
        [f]
        :=
        \overline{\{pf:p\in\mathbb C[z_1,\ldots,z_d]\}}^{\,\mathcal A}.
    \]
    We say that $f$ is cyclic in $\mathcal A$ if $[f]=\mathcal A$.
    Equivalently, $f$ is cyclic if there exists a sequence of polynomials
    $(p_n)$ such that
    \[
        p_nf\longrightarrow1
        \qquad\text{in }\mathcal A.
    \]

    The following two elementary observations will be crucial in the sequel. First, if
    $\mathcal A_2\subset\mathcal A_1$ continuously and $f$ is cyclic in
    $\mathcal A_2$, then $f$ is cyclic in $\mathcal A_1$. Indeed, if
    $p_nf\to1$ in $\mathcal A_2$, then the same convergence holds in
    $\mathcal A_1$. Thus, in a nested family of spaces, cyclicity propagates
    to larger spaces, whereas noncyclicity propagates to smaller spaces.
    Second, if $z_0$ is a bounded point evaluation for $\mathcal A$ and
    $f(z_0)=0$, then $f$ is not cyclic in $\mathcal A$.
\subsection{Overview of the unit disc}
    Characterizing cyclic vectors with respect to the shift operators on a Banach or Hilbert space of holomorphic functions is a classical problem in complex analysis and operator theory. Such problems have traditionally been challenging even in classical Hilbert
spaces of holomorphic functions on the unit disk $\mathbb{D}.$
Complete characterizations of cyclic vectors with respect to an operator are rare. An exception is the classical theorem of Beurling, which states that a function lying in the Hardy space \(H^2(\mathbb D)\) is cyclic
if and only if it is outer; see \cite{Beurling1949} for details. 

Let us now briefly describe the motivation for our work by mentioning related work on cyclicity in Dirichlet-type spaces of the unit disc.

For \(0\leq\alpha\leq1\), let
\[
    D_\alpha(\mathbb D)
    :=
    \left\{
        f(w)=\sum_{n\geq0}a_nw^n:
        \sum_{n\geq0}|a_n|^2(n+1)^\alpha<\infty
    \right\}.
\]
In particular, \(D_0(\mathbb D)=H^2(\mathbb D)\), while
\(D_1(\mathbb D)\) is the classical Dirichlet space. For
\(0<\alpha\leq1\), every function in \(D_\alpha(\mathbb D)\) admits
radial limits outside a set of \(\alpha\)-capacity zero; see,
for example, \cite{ElFallah2014_book}. Moreover, let $\mathcal{Z}(f)$ denote the zero set of a holomorphic function $f.$ Then, if $f\in D_{\beta}(\mathbb{D})$, the set
\(\mathcal{Z}(f)\cap\mathbb T\) is understood in terms of these boundary values.

In the classical Dirichlet space, Brown and Shields proved that if
\(f\) is cyclic, then \(f\) is outer and \(\mathcal{Z}(f)\cap\mathbb T\) has
logarithmic capacity zero. They conjectured that the converse also
holds \cite{BS84}. El-Fallah, Kellay, and Ransford proved this converse
when \(f\) extends continuously to \(\overline{\mathbb D}\) and
\(\mathcal{Z}(f)\cap\mathbb T\) is a generalized Cantor set
\cite{ElFallah_BS_Conj}. They subsequently extended this result to
weighted Dirichlet spaces, replacing logarithmic capacity by the Riesz 
\(\alpha\)-capacity \cite{ELFALLAH2010_Cantor}. For some more recent advancements on the topic, the reader is invited to also study the paper \cite{Vavitsas3}.
\subsection{Overview of the unit ball and bidisk}
In several complex variables, many works have considered the problem of cyclicity in Dirichlet-type spaces on the bidisk and the unit ball. Most of them focus on cyclicity polynomials. The reader is referred, among others, to the recent papers of Kosi\'nski-Vavitsas on cyclic polynomials in Dirichlet-type spaces on the unit ball, Hartz-Chalmoukis on cyclicity in the Drury-Arveson space on the unit ball, and Knese-Kosi\'nski-Ransford-Sola on cyclic polynomials in Dirichlet-type spaces on the bidisc, and the more recent work of the second named author on cyclicity of multipliers \cite{Kosinski_Vavitsas2023,ChalmoukisHartz2024,Knese2019_Aniso_Bidisk,Vavitsas2, Pouriya_SecondPaper}.

We denote the unit bidisk in $\mathbb{C}^2$ by
$$\mathbb{D}^2:=\{z=(z_1,z_2):|z_1|<1,|z_2|<1\}.$$
Our study focuses on the cyclicity of more general functions in Dirichlet-type spaces defined on this domain; specifically, we study when a function is \textit{critically cyclic}, a notion we shall explain in detail in the next subsection. 

 One should be precise with regard to the choice of the appropriate Dirichlet-type spaces on the bidisk. There are two different families of Dirichlet-type spaces that have been
considered in the literature. The more classical spaces, denoted by
\(\mathfrak D_\beta(\mathbb D^2)\) and introduced by Kaptano\u glu in \cite{Kaptanoglu1994}, are
invariant under composition with the automorphisms of \(\mathbb D^2\). More precisely, if $f\in \mathcal{O}(\mathbb{D}^2,\mathbb{C}),$ i.e., a holomorphic function on the bidisc with complex values, then it is known that this function has a Taylor expansion of the form
\[
    f(z_1,z_2)
    =
    \sum_{k,l\geq0}a_{kl}z_1^kz_2^l.
\]
The Dirichlet-type spaces introduced by Kaptano\u glu are defined as
\[ \mathfrak{D}_{\beta}(\mathbb{D}^2)=\Bigg\{f\in\mathcal{O}(\mathbb{D}^2,\mathbb{C}):
    \|f\|_{\mathfrak D_{\beta}}^2
    :=
    \sum_{k,l\geq0}
    |a_{kl}|^2(k+1)^\beta(l+1)^\beta<+\infty\Bigg\}.
\]
In \cite{Pouriya_FirstPaper, BesliSola,PoletskyStessin2008}, among other works, the authors introduced and studied a different class of Dirichlet-type spaces, denoted by \(\mathcal D_\beta(\mathbb D^2)\), and defined by a less punishing coefficient norm
\[
\mathcal D_{\beta}(\mathbb{D}^2)=\Bigg\{f\in\mathcal{O}(\mathbb{D}^2,\mathbb{C}):
    \|f\|_{\mathcal D_{\beta}}^2
    :=
    \sum_{k,l\geq0}
    |a_{kl}|^2(k+l+1)^\beta<+\infty\Bigg\}.
\]
In his paper \cite{Pouriya_FirstPaper}, the second named author defines these Dirichlet-type spaces in detail through the exhaustion function of the bidisk, providing another natural and effective definition of Dirichlet-type spaces on this domain.

For the range of parameters $\beta\in (0,2),$ both families of spaces possess some rather nice equivalent integral norms. More precisely, for the Dirichlet-type spaces defined by Kaptano\u glu, one has

\[\mathfrak{D}_{\beta}(\mathbb{D}^2):=\Biggl\{ f\in \mathcal{O}(\mathbb{D}^2):\int_{\mathbb{D}^2}\left|\frac{\partial^2(z_1z_2f(z_1,z_2))}{\partial z_1 \partial z_2}\right|^2dA_{\beta}(z_1)dA_{\beta}(z_2)<\infty\Biggr\},\]
where $dA_{\beta}(z)=(1-|z|^2)^{1-\beta}dA(z),$ and $dA(z)=\frac{1}{\pi}dxdy,$ for $z=x+iy.$
For the Poletsky-Stessin Dirichlet-type spaces, one also has the following equivalent integral norm
\begin{equation} \label{eq: Integral Norm Defn}
\|f\|_{\mathcal D_\beta(\mathbb D^2)}^2
=
|f(0)|^2
+
\int_0^1
\int_{\mathbb T^2}
|Rf(\rho\zeta)|^2
\,dm_2(\zeta)\,
(1-\rho)^{1-\beta}\,d\rho ,
\end{equation}
where $dm_2=\frac{1}{(2\pi)^2}d\theta_1d\theta_2$
Here, by $Rf$ we denote the \textit{radial derivative} of a holomorphic function $f\in\mathcal{O}(\mathbb{D}^2,\mathbb{C})$, which is defined as follows:
    \begin{equation*}
        R f (z) = z_1 \frac{\partial f}{\partial z_1} (z)+z_2 \frac{\partial f}{\partial z_2} (z).
    \end{equation*}
The cyclic polynomials in \(\mathfrak D_\beta(\mathbb D^2)\) were
completely characterized in
\cite{Beneteau2016_TwoVariables,Knese2019_Aniso_Bidisk}. In these
spaces, cyclicity is not determined only by the dimension of the
boundary zero set. Indeed, for \(1/2<\beta\leq1\), the polynomials
\(1-z_1\) and \(1-z_2\) are cyclic, although their zero sets on
\(\mathbb T^2\) have dimension one, whereas a polynomial depending on
both variables with a one-dimensional boundary zero set need not be
cyclic.

\subsection{Critical cyclicity and main result}
For \(\mathcal D_\beta(\mathbb D^2)\), the results of
\cite{Pouriya_FirstPaper,NailwalZalar2026} provide a complete
characterization. More precisely, let \(p\) be an irreducible polynomial that does not
vanish in \(\mathbb D^2\). If \(\beta\leq1\), then \(p\) is cyclic in
\(\mathcal D_\beta(\mathbb D^2)\). If \(1<\beta\leq2\), then \(p\) is cyclic if
and only if \(\mathcal{Z}(p)\cap\mathbb T^2\) is empty or finite. Finally, if
\(\beta>2\), then \(p\) is cyclic if and only if
\(\mathcal{Z}(p)\cap\mathbb T^2=\varnothing\). Thus, for \(\mathcal D_\beta(\mathbb D^2)\),
the polynomial results suggest that cyclicity depends heavily on the
dimension of the boundary zero set.

Let \(E\subset\mathbb T^2\) be a closed set and assume further that there
exists a \(\beta_c\in(0,2]\) such that
$
    \operatorname{cap}_\beta(E)=0
    \Longleftrightarrow
    \beta\leq\beta_c.$
We call \(f\in \mathcal D_2(\mathbb D^2)\) \textit{critically cyclic} for \(E\) if
$
    \mathcal{Z}(f)\cap\mathbb T^2=E
$
and \(f\) is cyclic in \(\mathcal D_\beta(\mathbb D^2)\) if and only if
\(\beta\leq\beta_c\). The number \(\beta_c\) is called the \textit{critical
index} of \(f\). The examples constructed in
\cite{Pouriya_CriticalBall} suggest that, on the unit ball, the
critical index is determined by the Kor\'anyi--Hausdorff dimension of
the boundary zero set. Our main result is the following.
   
    \begin{thm} \label{MainThm1}
        Given $\beta_c \in (0,2]$, there exists $f \in \mathcal D_2(\D^2)$ such that $f$ is cyclic in $\mathcal D_\beta(\D^2)$ if and only if $\beta \leq \beta_c$.
    \end{thm}

The critical index in the bidisk case satisfies
$$\beta_c=2-\mathrm{dim}_H(\mathcal{Z}(f)\cap \mathbb{T}^2),$$ where $\mathrm{dim}_H$ denotes the Hausdorff dimension of the zero set of the function $f$ intersected with the two-torus.

The noncyclicity assertions in both cases follow from capacity
arguments. For \(0<\beta_c\leq1\), cyclicity is proved by the radial
dilation method. Specifically we modify ideas of the second named author from his work on critically critically cyclic functions on the unit ball, see \cite{Pouriya_CriticalBall}. Unfortunately, when \(1<\beta_c\leq2\), this method does not apply. Indeed, it does not even settle the cyclicity of the polynomial
\(2-z_1-z_2\) for \(3/2<\beta\leq2\). This question was posed in
\cite[Open Problem~1]{Pouriya_FirstPaper} and subsequently answered
affirmatively by Nailwal and Zalar \cite{NailwalZalar2026}. In the
second case, we instead lift a one-variable analytic deformation to
the bidisk. This method also gives an alternative proof of the
cyclicity of \(2-z_1-z_2\). One also has to note here that our result does not directly translate to the $\mathfrak{D}_{\beta}-$spaces defined by Kaptano\u glu. From Proposition 4.5. of \cite{BesliSola}, we know that $\mathcal{D}_{\beta}=\mathfrak{D}_{(0,\beta)}\cap\mathfrak{D}_{(\beta,0)}.$ This fact alone does not imply that our result holds for $\mathfrak{D_{\beta}}-$spaces. 

\section{Preliminaries} \label{Section2}
\subsection{Basics and notation} \label{Subsec2.1}

    	For $z=(z_1,\dots,z_n)$ and $w=(w_1,\dots,w_n)$ in
$\mathbb C^n$, we denote by
\[
    \langle z,w\rangle
    :=\sum_{j=1}^n z_j\overline{w_j}
\]
the standard Hermitian inner product on $\mathbb C^n$. We write
$
    |z|=\sqrt{\langle z,z\rangle}
$
for the associated Euclidean norm. We also follow the standard notations in the literature for the unit disc, circle and ball respectively,
    $$\mathbb D:=\{z\in\mathbb C:|z|<1\},
    \qquad
    \mathbb T:=\partial\mathbb D,
    \qquad
    \mathbb B_n:=\{z\in\mathbb C^n:|z|<1\}.$$
The following two remarks are rather simple observations, but at the same time, they play crucial roles in our proofs.
\begin{rem}\label{rem:radial-derivative}
    Let $f \in \mathcal{O}(\D^2)$ and
    $f(z)= \sum_{k,l \geq 0} a_{kl} z_1^k z_2^l$. Then
    $Rf(z) = \sum_{k,l \geq0} (k+l) a_{kl} z_1^k z_2^l$. We have the following equivalent norm relation,
    \begin{equation}
        \|f\|_\beta
        \asymp |f(0)|+\|Rf\|_{\beta-2}.
        \label{eq:radial-norm-equivalence}
    \end{equation}
    Hence, $f \in \mathcal D_\beta(\D^2)$ if and only if
    $Rf \in \mathcal D_{\beta-2}(\D^2)$.
\end{rem}

\begin{rem}\label{rem:monotonicity}
Let $\beta_1\leq\beta_2$. Then $D_{\beta_2}(\mathbb D^2)\subset D_{\beta_1}(\mathbb D^2)$ continuously. Consequently, if $f$ is cyclic in $D_{\beta_2}(\mathbb D^2)$, then it is cyclic in $D_{\beta_1}(\mathbb D^2)$.
\end{rem}

\subsection{Notational conventions}
    We write $A \lesssim B$ to indicate that there exists a constant $C > 0$ such that $A \leq C B$. If both $A \lesssim B$ and $B \lesssim A$ hold, we write $A \asymp B$. If the constant depends on a parameter $\alpha,$ then we write $A \lesssim_{\alpha}B.$ Similarly, we write $\asymp_{\alpha}.$ Moreover, assume that $\Omega\subset\mathbb C^n$ is a bounded domain. For $k\in\mathbb N_0\cup\{\infty\}$, we define
\[
    A^k(\Omega):=\mathcal O(\Omega)\cap \mathcal{C}^k(\overline\Omega),
\]
where $\mathcal{C}^k$ denotes the class of functions with continuous partial derivatives of order $k$.
Finally, for a finite Borel measure $\sigma$ on $\mathbb T$, we write
\[
    \widehat\sigma(j):=\int_{\mathbb T}\zeta^j\,d\sigma(\zeta),
    \qquad j\in\mathbb Z.
\]

\subsection{Tools for the proofs}

In the present section we briefly present the basic theorems, lemmas, and propositions that will be applied in the proofs of the main results. When necessary, we provide their proofs. The first is a practical sufficient condition for the cyclicity of a function. For a detailed proof, one may consult \cite{Knese2019_Aniso_Bidisk}.

\begin{thm} \label{Radial_Dilation}
    Let $f \in \mathcal D_{\beta}(\D^2)$ and assume that $f$ does not vanish on $\D^2$. For $0 < r < 1$, define the radial dilation
    $
        f_r (z) = f(rz).
    $
    If
    \begin{equation*}
        \sup_{0 < r < 1} \left\| \frac{f}{f_r} \right\|_{\beta} < +\infty,
    \end{equation*}
    then $f$ is cyclic in $\mathcal D_{\beta}(\D^2)$.
\end{thm}

For the proof of the main result, we shall need the reproducing kernel of the $\mathcal D_{\beta}(\mathbb{D}^2)$ spaces.

\begin{prop} \label{prop: kernel}
Let $\beta \in (0,2)$. Then $\mathcal D_\beta(\mathbb{D}^2)$ is an RKHS (reproducing kernel Hilbert space), i.e., it can be endowed with an inner product for which the reproducing kernel $k_\beta$ is given by
\begin{equation*}
    k_{\beta}(z,w) = \int_0^1 \frac{dt}{\left(1 - t z_1 \overline{w}_1 - (1-t) z_2 \overline{w}_2 \right)^{2-\beta}}, 
    \quad z,w \in \mathbb{D}^2.
\end{equation*}
\end{prop}
\begin{proof}
    A straightforward application of Stirling's formula yields
\begin{equation*}
    (m+1)^{\beta} \asymp \frac{\Gamma(m+1)\,\Gamma(2-\beta)}{\Gamma(2-\beta + m)} \, (m+1).
\end{equation*}
As a consequence, for $f(z) = \sum_{k,l \ge 0} a_{kl} z_1^k z_2^l \in \mathcal D_\beta(\mathbb{D}^2)$, we have
\begin{equation} \label{eq:equivalent_norm}
    \| f \|_\beta^2 \asymp \sum_{k,l \ge 0} 
    \frac{\Gamma(k+l+1)\,\Gamma(2-\beta)}{\Gamma(2-\beta + k + l)} \, (k+l+1) |a_{kl}|^2.
\end{equation}
Using the right-hand side of \eqref{eq:equivalent_norm} as the definition of the norm, we compute $k_\beta$ as follows:
\begin{equation*}
\begin{aligned}
k_{\beta}(z,w)
&= \sum_{k,l \ge 0} 
\frac{\Gamma(2-\beta + k + l)}{\Gamma(k+l+1)\,\Gamma(2-\beta)\,(k+l+1)} 
(z_1 \overline{w}_1)^{k} (z_2 \overline{w}_2)^{l} \\
&= \sum_{m=0}^{\infty} 
\frac{\Gamma(2-\beta + m)}{\Gamma(m+1)\,\Gamma(2-\beta)\,(m+1)} 
\sum_{j=0}^{m} (z_1 \overline{w}_1)^j (z_2 \overline{w}_2)^{m-j} \\
&= \sum_{m=0}^{\infty} 
\frac{\Gamma(2-\beta + m)}{\Gamma(m+1)\,\Gamma(2-\beta)\,(m+1)} 
\frac{(z_1 \overline{w}_1)^{m+1} - (z_2 \overline{w}_2)^{m+1}}
     {z_1 \overline{w}_1 - z_2 \overline{w}_2}.
\end{aligned}
\end{equation*}
Using the identity
\[
\frac{x_1^{m+1} - x_2^{m+1}}{x_1 - x_2}
= (m+1) \int_0^1 (t x_1 + (1-t) x_2)^m \, dt,
\]
together with Fubini's theorem, we obtain
\begin{equation*}
\begin{aligned}
k_{\beta}(z,w)
&= \int_0^1 \sum_{m=0}^{\infty} 
\frac{\Gamma(2-\beta + m)}{\Gamma(m+1)\,\Gamma(2-\beta)} 
\left( t z_1 \overline{w}_1 + (1-t) z_2 \overline{w}_2 \right)^m \, dt \\
&= \int_0^1 
\frac{dt}{\left(1 - t z_1 \overline{w}_1 - (1-t) z_2 \overline{w}_2 \right)^{2-\beta}}.
\end{aligned}
\end{equation*}
\end{proof}
\begin{rem}\label{rem:kernel-beta-two}
    In a similar fashion, one may equip $\mathcal D_2(\D^2)$ with an inner product such that
    \begin{equation*}
        k_{2}(z,w) = \int_0^1 \log\left(\frac{e}{1 - t z_1 \overline{w}_1 - (1-t) z_2 \overline{w}_2 }\right)  dt, 
    \quad z,w \in \mathbb{D}^2.
    \end{equation*}
\end{rem}

\begin{lem}\label{lem:segment-integral-estimates}
Let $x,y\in\mathbb C$ and set $M:=\max\{|x|,|y|\}>0.$
For $0<\alpha<1$,
\begin{equation}\label{eq:affine-power-estimate}
    \int_0^1|(1-t)x+t y|^{-\alpha}\,dt
    \lesssim_\alpha M^{-\alpha}.
\end{equation}
Moreover, if $M\leq1$, then
\begin{equation}\label{eq:affine-log-estimate}
    \int_0^1\log\frac{e}{|(1-t)x+ty|}\,dt
    \lesssim\log\frac{e}{M}.
\end{equation}
\end{lem}

\begin{proof}
Without loss of generality,
we may assume that
    $x=M>0$,
and consider $x+t(y-x)$. Set
$v=\frac{y-x}{x}.$
Then $|v|\leq2$ and
$
    |1+v|=\frac{|y|}{x}\leq1.
$
Hence,
\[
    |x+t(y-x)| \geq x |1+t  v|
    \geq M |1+t \operatorname{Re} v|.
\]

The result now follows from the integrability of $|1- t \operatorname{Re} v |^{-\alpha}$ and $\log(e/|1- t \operatorname{Re} v|)$, for $\operatorname{Re}v\in[-2,2]$.
\end{proof}

\begin{lem}\label{lem:kernel-estimates}
    Let $1<\beta\leq 2$. Then, for all $z,w\in \overline{\D}^2$, we have
    \begin{equation*}
        \operatorname{Re} k_\beta(z,w) \asymp |k_\beta(z,w)|.
    \end{equation*}
    Moreover, if $z,w\in \mathbb{T}^2$ and $1<\beta<2$, then
    \begin{equation*}
        |k_\beta(z,w)| \asymp \frac{1}{|z-w|^{2-\beta}}.
    \end{equation*}
    In the case $\beta=2$, we have
    \begin{equation*}
        |k_2(z,w)| \asymp \log \frac{e \sqrt{2}}{|z-w|}.
    \end{equation*}
\end{lem}
\begin{proof}
    The first part of the statement is an immediate consequence of Proposition~\ref{prop: kernel}. 

        We now assume that \(z,w\in\mathbb T^2\) and \(z\neq w\).
    Since the kernel is invariant under coordinate rotations (i.e., rotating each coordinate separately), we may
    assume that
    $w=(1,1)$, $z=(e^{i\theta},e^{i\varphi})$ and $\theta,\varphi\in[-\pi,\pi].$    When \(\theta\) and \(\varphi\) are sufficiently small,
    \begin{equation}
        |z-w|\asymp \max\{|\theta|,|\varphi|\}.
        \label{eq:angle-distance}
    \end{equation}

    Let \(1<\beta<2\) and put \(\alpha=2-\beta\). Then,
    \begin{equation}
        |k_\beta(z,w)|
        \asymp
        \int_0^1
        \left|
            t(1-e^{i\theta})+(1-t)(1-e^{i\varphi})
        \right|^{-\alpha}\,dt.
        \label{eq:kernel-power-integral}
    \end{equation}
    Uniformly in \(t\),
    $
        \left|
            t(1-e^{i\theta})+(1-t)(1-e^{i\varphi})
        \right|
        \lesssim\max\{|\theta|,|\varphi|\},
    $
    which gives the required lower bound in
    \eqref{eq:kernel-power-integral}. On the other hand,
    \[
        \left|
            t(1-e^{i\theta})+(1-t)(1-e^{i\varphi})
        \right|
        \geq
        |t\sin\theta+(1-t)\sin\varphi|.
    \]
    Since
    \[
        \max\{|\sin\theta|,|\sin\varphi|\}
        \asymp\max\{|\theta|,|\varphi|\},
    \]
    estimates \eqref{eq:affine-power-estimate} and
    \eqref{eq:angle-distance} give
    \[
        |k_\beta(z,w)|
        \asymp
        \frac{1}{|z-w|^{2-\beta}}.
    \]

    For \(\beta=2\), the same
    argument, using \eqref{eq:affine-log-estimate}, gives
    \[
    \begin{aligned}
        |k_2(z,w)|
        &\asymp
        \int_0^1
        \log\frac{e}{
        \left|t(1-e^{i\theta})+(1-t)(1-e^{i\varphi})\right|}
        \,dt \\
        &\asymp
        \log\frac{e}{\max\{|\theta|,|\varphi|\}}
        \asymp
        \log\frac{e\sqrt2}{|z-w|}.
    \end{aligned}
    \]
\end{proof}
\begin{defn}
Let $E \subset  \mathbb{T}^2$ be a Borel set and let $\mathcal P(E)$ denote the set of all
Borel probability measures supported on $E$.  
For $\beta \in (0,2]$ and $\mu \in \mathcal P(E)$, the $\beta$-energy of $\mu$ is
\begin{equation*}
    I_{\beta}[\mu]
    =
    \iint_{\mathbb{T}^2 \times \mathbb{T}^2} k_{\beta}\!\left( \zeta,\eta \right)
    \, d\mu(\zeta)\, d\mu(\eta).
\end{equation*}
The \emph{$\beta$-capacity} of $E$ is then defined by
\begin{equation*}    
    \operatorname{cap}_{\beta}(E)
    =
    \frac{1}{\inf \{ I_{\beta}[\mu] : \mu \in \mathcal P(E) \}}.
\end{equation*}
\end{defn}
The following is a useful tool for determining the non-cyclicity of a function.

\begin{thm} \label{thm:cap_noncyclic}
    Suppose $f \in \mathcal D_{\beta}(\D^2) \cap \mathcal{C} (\overline{\D^2})$ satisfies $\operatorname{cap}_{\beta}(\mathcal{Z}(f) \cap \mathbb{T}^2) > 0$. 
    Then $f$ is not cyclic in $\mathcal D_\beta(\D^2)$.
\end{thm}

The proof follows directly the techniques of \cite{BS84}, similar to \cite{Knese2019_Aniso_Bidisk}.

The following bounded point-evaluation property is proved in
\cite[Proposition~34, p.~14]{Pouriya_FirstPaper}.
\begin{lem}\label{lem:bounded-point-evaluations}
Let $\beta>2$. Then every $f\in \mathcal D_\beta(\mathbb D^2)$ extends
continuously to $\overline{\mathbb D^2}$ and, for every
$z_0\in\overline{\mathbb D^2}$, the point evaluation functional
$\Lambda_{z_0}:\mathcal D_\beta(\mathbb D^2)\longrightarrow\mathbb C,$ $\Lambda_{z_0}(f):=f(z_0),$
is bounded.
\end{lem}

The following consequence of bounded point evaluations can be found in
\cite[Theorem~14 and Remark~15, p.~7]{Pouriya_FirstPaper}.
\begin{lem}\label{lem:zero-at-bounded-point-evaluation}
Let $z_0\in\overline{\mathbb D^2}$ be a bounded point evaluation for
$\mathcal D_\beta(\mathbb D^2)$. If $f\in \mathcal D_\beta(\mathbb D^2)$ and $f(z_0)=0$
then $f$ is not cyclic in $\mathcal D_\beta(\mathbb D^2)$.
\end{lem}

In the proof of cyclicity, we shall invoke the following lemma, which is inspired by
\cite[Lemma~5.3]{AlemanPerfektRichter}.

\begin{lem}
    Let \(0<\beta\leq 2\), and let
    \(\phi,\psi\in\mathcal{M}(\mathcal D_\beta(\mathbb{D}^2))\). If $\frac{\phi}{\psi}\in H^\infty(\mathbb{D}^2)$,
    then $\frac{\phi^2}{\psi}
        \in\mathcal{M}(\mathcal D_\beta(\mathbb{D}^2))$.
    In particular, $\phi^2\in[\psi]$.
    
    \label{Perfekt}
\end{lem}

\begin{proof}
    We write
    $g\in\mathcal{M}(\mathcal D_\beta,\mathcal D_{\beta-2})$
    if multiplication by \(g\) defines a bounded operator from
    \(\mathcal D_\beta(\mathbb{D}^2)\) into \(\mathcal D_{\beta-2}(\mathbb{D}^2)\).

    If \(\theta\in\mathcal{M}(\mathcal D_\beta(\mathbb{D}^2))\), then
    \begin{equation}
        R\theta\in\mathcal{M}(\mathcal D_\beta,\mathcal D_{\beta-2}).
        \label{eq:radial-multiplier-property}
    \end{equation}
    Indeed, for \(u\in \mathcal D_\beta(\mathbb{D}^2)\),
    \[
        (R\theta)u=R(\theta u)-\theta Ru.
    \]
    The first term belongs to \(\mathcal D_{\beta-2}(\mathbb{D}^2)\) because
    \(\theta u\in \mathcal D_\beta(\mathbb{D}^2)\), while the second belongs to
    \(\mathcal D_{\beta-2}(\mathbb{D}^2)\) because
    $\theta\in H^\infty(\mathbb{D}^2)
        =\mathcal{M}(\mathcal D_{\beta-2}(\mathbb{D}^2)).$
    Set $h:=\frac{\phi}{\psi}$ and $\Psi:=\frac{\phi^2}{\psi}=\phi h.$
    Then \(\Psi\in H^\infty(\mathbb{D}^2)\), and
    \[
        R\Psi=2hR\phi-h^2R\psi.
    \]
    By \eqref{eq:radial-multiplier-property},
    $R\phi,R\psi\in\mathcal{M}(\mathcal D_\beta,\mathcal D_{\beta-2})$. Since
    \(h,h^2\in H^\infty(\mathbb{D}^2)\), it follows that
    \[
        R\Psi\in\mathcal{M}(\mathcal D_\beta,\mathcal D_{\beta-2}).
    \]
    For $u \in \mathcal D_\beta(\mathbb{D}^2)$, we have
    \[
        R(\Psi u)=(R\Psi)u+\Psi Ru.
    \]
    Thus, $R(\Psi u)\in \mathcal D_{\beta-2}(\mathbb{D}^2)$.
    Remark~\ref{rem:radial-derivative}, in the form
    \eqref{eq:radial-norm-equivalence}, gives
    $\Psi u\in \mathcal D_\beta(\mathbb{D}^2)$ and
    \[
        \Psi\in\mathcal{M}(\mathcal D_\beta(\mathbb{D}^2)).
    \]
    Finally,
    $
        \phi^2=\psi\Psi.
    $
    This implies that
    \(\phi^2\in[\psi]\).
\end{proof}

\section{Proof of Theorem~\ref{MainThm1}} \label{Section3} Let us explain the structure of the proof. First, we collect the constructions that will be used in the proof. Then, we consider separately the cases $\beta_c\in(0,1]$ and
$\beta_c\in(1,2]$. In the first case, we use the radial dilation
method. In the second case, we lift a one-variable analytic deformation
to the bidisk. Before we proceed with the case $\beta_c\in(0,1],$ we provide a brief explanation of the main idea that leads to the proof.

\subsection{From the unit ball to the bidisk} Here we briefly describe the main idea of our proof for the case $\beta_c\in(0,1].$ 
First, let us recall the notion of a $\mathcal K$-set, introduced by Bruna and
Ortega in the paper \cite{Bruna1986}. On $\partial\mathbb B_n$, we use the Kor\'anyi
pseudodistance
\[
    d_K(\zeta,\eta)
    :=\bigl|1-\langle\zeta,\eta\rangle\bigr|,
    \qquad \zeta,\eta\in\partial\mathbb B_n,
\]
and denote by
\[
    K(\zeta,r)
    :=\bigl\{\eta\in\partial\mathbb B_n:
                   d_K(\zeta,\eta)<r\bigr\}
\]
the corresponding Kor\'anyi ball.

\begin{defn}
Let $E\subset\partial\mathbb B_n$ be closed. For
$A\subset\partial\mathbb B_n$, let $N_\varepsilon(A)$ denote the
smallest number of Kor\'anyi balls of radius $\varepsilon$ needed to
cover $A$. We say that $E$ is a \emph{$\mathcal K$-set} if there exists a
constant $C>0$ such that
\[
    \int_0^r N_\varepsilon\bigl(E\cap K(\zeta,r)\bigr)\,d\varepsilon
    \leq Cr
\]
for every $\zeta\in\partial\mathbb B_n$ and every $0<r\leq1$.
\end{defn}

Suppose that $\Gamma\subset\partial\mathbb B_n$ is a smooth compact
transversal curve with a smooth regular parametrization
$\gamma:\mathbb T\to\Gamma$. Then
\[
    d_K\bigl(\gamma(s),\gamma(t)\bigr)\asymp|s-t|
\]
locally and uniformly on $\mathbb T$. Consequently, the Kor\'anyi
covering numbers of a closed set $E\subset\Gamma$ are comparable to the
ordinary covering numbers of $\gamma^{-1}(E)$; see for example
\cite[Lemma~4.1]{Bruna1986}.

Consider the diagonal
$
    \Delta
    :=\bigl\{(e^{it},e^{it}):t\in\mathbb R\bigr\}
    \subset\mathbb T^2
    \subset\partial\mathbb B_2(0,\sqrt2),
$
parametrized by $\gamma(t)=(e^{it},e^{it})$. Since
$
    \bigl\langle\gamma'(t),\gamma(t)\bigr\rangle=2i\neq0,
$
the curve $\Delta$ is transversal to the complex tangent bundle of
$\partial\mathbb B_2(0,\sqrt2)$.

With a slight abuse of notation, we also denote by $d_K$ the rescaled
Kor\'anyi pseudodistance on
$
    \overline{\mathbb B_2(0,\sqrt2)}
    \times\partial\mathbb B_2(0,\sqrt2),
$
defined by
\[
    d_K(z,\zeta)
    :=\left|1-\frac{\langle z,\zeta\rangle}{2}\right|.
\]
For a closed set $A\subset\partial\mathbb B_2(0,\sqrt2)$, we write
\[
    d_K(z,A):=\inf_{\zeta\in A}d_K(z,\zeta).
\]
We use the same terminology for $\mathcal K$-sets on
$\partial\mathbb B_2(0,\sqrt2)$, after the dilation
$z\mapsto z/\sqrt2$.

We shall use two constructions on $\mathbb B_2(0,\sqrt2)$.

\begin{prop}\label{prop:ball-construction-E}
Let $E\subset\Delta$ be a $\mathcal K$-set. Then there exists
$H_E\in A^1\bigl(\mathbb B_2(0,\sqrt2)\bigr)$
such that
\[
    |H_E(z)|\asymp d_K(z,E),
    \qquad z\in\overline{\mathbb B_2(0,\sqrt2)},
\]
and, for every integer $j\geq1$,
\[
    |R^jH_E(z)|
    \lesssim d_K(z,E)^{1-j},
    \qquad z\in\mathbb B_2(0,\sqrt2).
\]
In particular, if
$h_E:=H_E|_{\mathbb D^2},$
then $h_E$ is nonvanishing on $\mathbb D^2$ and
$ \mathcal Z(h_E)\cap\mathbb T^2=E.$
\end{prop}

\begin{proof} To see this, apply the construction of Bruna and Ortega
\cite[Theorem~4.3]{Bruna1986}, followed by the dilation
$z\mapsto z/\sqrt2$. The result for $h_E$ follows by restriction to
$\mathbb D^2$ (see also \cite[Proposition~1]{Pouriya_CriticalBall}).
\end{proof}

Fix a sufficiently small number $\delta>0$ and define the
complex-tangential thickening or enlargement of the set $E$ by
\[
    \widetilde E
    :=\left\{
       (e^{i(u+v)},e^{i(u-v)}):
       (e^{iu},e^{iu})\in E,\ -\delta\leq v\leq\delta
      \right\}.
\]

\begin{prop}\label{prop:ball-construction-Etilde}
Let $E\subset\Delta$ be a $\mathcal K$-set and let $\widetilde E$ be
defined as above. Then there exist $\ell>0$ and a function
$H_{\widetilde E}$, holomorphic on $\mathbb B_2(0,\sqrt2)$, such that
\[
    |H_{\widetilde E}(z)|
    \asymp d_K(z,\widetilde E)^\ell,
    \qquad z\in\overline{\mathbb B_2(0,\sqrt2)},
\]
and, for every integer $j\geq1$,
\[
    |R^jH_{\widetilde E}(z)|
    \lesssim d_K(z,\widetilde E)^{\ell-j},
    \qquad z\in\mathbb B_2(0,\sqrt2).
\]
In particular, if
$
    h_{\widetilde E}
    :=H_{\widetilde E}|_{\mathbb D^2},
$
then $h_{\widetilde E}$ is nonvanishing on $\mathbb D^2$ and
$
    \mathcal Z(h_{\widetilde E})\cap\mathbb T^2=\widetilde E.
$
\end{prop}

\begin{proof}
Here we may use the construction of Chaumat and Chollet in
\cite[Proposition~17]{Chaumat_Chollet_Hausdorff}, together with the refined
estimates in \cite[Proposition~27]{Pouriya_CriticalBall}. We may take
\[
    H_{\widetilde E}(z)
    :=\exp\left(-\psi\left(\frac{z}{\sqrt{2}}\right)\right),
\]
where $\psi$ is as in \cite[Proposition~27]{Pouriya_CriticalBall}.

\end{proof}

\begin{prop}\label{prop:ball-restrictions-multipliers}
Let $0<\beta\leq2$. If $\varphi,R\varphi\in
H^\infty(\mathbb D^2)$, then
$
    \varphi\in\mathcal M(\mathcal D_\beta(\mathbb D^2)).
$
In particular, the functions $h_E$ and $h_{\widetilde E}$ belong to
$\mathcal D_2(\mathbb D^2)$ and are multipliers of $\mathcal D_\beta(\mathbb D^2)$ for
every $0<\beta\leq2$.
\end{prop}

\begin{proof}
For $u\in \mathcal D_\beta(\mathbb D^2)$, we have
$
    R(\varphi u)=(R\varphi)u+\varphi Ru.
$
For $\beta\leq0$, bounded holomorphic functions are multipliers of
$\mathcal D_\beta(\mathbb D^2)$. Indeed, this follows from the equivalent $\mathcal D_{\beta}$-norm involving the radial derivative, and from the
usual multiplier property of $H^\infty(\mathbb D^2)$ on
$H^2(\mathbb D^2)$ when $\beta=0$. Since
$\mathcal D_\beta(\mathbb D^2)\subset \mathcal D_{\beta-2}(\mathbb D^2)$, one has
\[
    R(\varphi u)\in \mathcal D_{\beta-2}(\mathbb D^2).
\]
Therefore, \eqref{eq:radial-norm-equivalence} gives
$\varphi\in\mathcal M(\mathcal D_\beta(\mathbb D^2))$.
\end{proof}

\subsection{Proof of Case 1:
\texorpdfstring{$0<\beta_c\leq1$}{0 < beta\_c <= 1}}
Set $d:=1-\beta_c$. Choose a $\mathcal K$-set
$E\subset\Delta$ of Hausdorff dimension $d$ such that
\begin{equation}
    |E_t|\asymp t^{1-d},
    \label{eq:case1-neighborhood-growth}
\end{equation}
where $E_t$ denotes the $t$-neighborhood of $E$ inside $\Delta$ and
$|\cdot|$ denotes one-dimensional Lebesgue measure. Such sets may be
obtained, for instance, from Cantor sets with constant dissection
ratio; see \cite{ELFALLAH2010_Cantor}. Let
$\widetilde E$ and
$h_{\widetilde E}$ be given by
Proposition~\ref{prop:ball-construction-Etilde}.

Denote by
$
    E^1:=\pi_1(E)
    =\{e^{it}:(e^{it},e^{it})\in E\}$
the projection of $E$ onto the first variable, and let $m_\delta$ be
the uniform probability measure on
$
    J_\delta:=\{e^{iv}:-\delta\leq v\leq\delta\}.
$
Define
$
    \Phi(e^{iu},e^{iv})
    :=(e^{i(u+v)},e^{i(u-v)}).
$

\begin{prop}\label{prop:case1-noncyclicity}
The function $h_{\widetilde E}$ is not cyclic in
$\mathcal D_\beta(\mathbb D^2)$ whenever
$
    \beta>\beta_c=1-d.
$
\end{prop}

\begin{proof}
Fix
$
    \beta_c<\beta <2.
$
As $\dim_{\mathrm H}E^1=d$ and $\beta>1-d$, the capacity criterion in
\cite[Section~1 and Theorem~4.3]{ELFALLAH2010_Cantor} provides a
probability measure $\nu$ supported on $E^1$ such that
\[
    \sum_{m\geq0}
    \frac{|\widehat\nu(m)|^2}{(m+1)^\beta}<\infty.
\]

Let
$
    \mu=\Phi_*(\nu\otimes m_\delta)
    \in\mathcal P(\widetilde E).
$
denote the push-forward measure. 
For $0\leq k\leq m$, we have
\[
    \int_{\mathbb T^2}
    \zeta_1^k\zeta_2^{m-k}\,d\mu(\zeta)
    =
    \widehat\nu(m)\widehat m_\delta(2k-m).
\]
Moreover,
$
    \widehat m_\delta(j)
    =\frac{\sin(j\delta)}{j\delta}$, $j\neq0,
$
and $\widehat m_\delta(0)=1$. It follows that
\[
\begin{aligned}
    \sum_{k=0}^m
    \left|\widehat m_\delta(2k-m)\right|^2
    \lesssim
    1+\frac{1}{\delta^2}\sum_{j\geq1}\frac{1}{j^2}
    =:C_\delta.
\end{aligned}
\]
Consequently,
\[
\begin{aligned}
    I_\beta[\mu]
    &\asymp
    \sum_{m\geq0}\frac{1}{(m+1)^\beta}
    \sum_{k=0}^m
    \left|
    \int_{\mathbb T^2}
    \zeta_1^k\zeta_2^{m-k}\,d\mu(\zeta)
    \right|^2 \\
    &=
    \sum_{m\geq0}
    \frac{|\widehat\nu(m)|^2}{(m+1)^\beta}
    \sum_{k=0}^m
    \left|\widehat m_\delta(2k-m)\right|^2 \\
    &\lesssim
    \sum_{m\geq0}
    \frac{|\widehat\nu(m)|^2}{(m+1)^\beta}
    <\infty.
\end{aligned}
\]
Thus
$
    \operatorname{cap}_\beta(\widetilde E)>0,
$
and Theorem~\ref{thm:cap_noncyclic} shows that
$h_{\widetilde E}$ is not cyclic in $\mathcal D_\beta(\mathbb D^2)$. Since $\beta\in(\beta_c,2)$ was arbitrary,
Remark~\ref{rem:monotonicity} gives the conclusion for every
$\beta>\beta_c$.
\end{proof}

\begin{prop}\label{prop:case1-cyclicity}
The function $h_{\widetilde E}$ is cyclic in
$\mathcal D_\beta(\mathbb D^2)$ whenever
$
    \beta\leq\beta_c=1-d.
$
\end{prop}

\begin{proof}
It is enough to prove cyclicity for $\beta=\beta_c$. Choose an integer
$N$ such that $N\ell>2$. Since
$[h_{\widetilde E}^N]\subset[h_{\widetilde E}],$
it is enough to prove that $h_{\widetilde E}^N$ is cyclic. Set
$h:=h_{\widetilde E}^N$ and $h_r(z):=h(rz)$. The estimates of
Proposition~\ref{prop:ball-construction-Etilde} remain valid with
$\ell$ replaced by $N\ell$, so we may assume that $\ell>2$.

For $0<r\leq1/2$, the functions $1/h_r$ and their first radial
derivatives are uniformly bounded. We may therefore assume that
$1/2<r<1$.

For $z\in\overline{\mathbb D^2}$ and $\eta\in\widetilde E$, set
$w:=\langle z,\eta\rangle/2$. Since $|w|\leq1$, we have
\[
\begin{aligned}
    |1-rw|^2
    &=(1-r)^2+2r(1-r)\operatorname{Re}(1-w)
      +r^2|1-w|^2 \\
    &\geq r^2|1-w|^2.
\end{aligned}
\]
Hence, for $1/2<r<1$,
\[
    d_K(rz,\widetilde E)
    \geq\frac12d_K(z,\widetilde E),
    \qquad
    1-r\leq d_K(rz,\widetilde E).
\]
Hence,
\[
    d_K(z,\widetilde E)+(1-r)
    \lesssim d_K(rz,\widetilde E).
\]
Integrating
the estimates of Proposition~\ref{prop:ball-construction-Etilde}
along the line segment connecting $rz$ to $z$ gives
\[
\begin{aligned}
    |h-h_r|(z)
    &\lesssim
    (1-r)d_K(rz,\widetilde E)^{\ell-1},\\
    |Rh-Rh_r|(z)
    &\lesssim
    (1-r)d_K(rz,\widetilde E)^{\ell-2}.
\end{aligned}
\]
Consequently,
\begin{equation}
\begin{aligned}
    \left|R\left(\frac{h}{h_r}\right)\right|(z)
    &=
    \left|
    \frac{(h_r-h)Rh+h(Rh-Rh_r)}{h_r^2}
    \right|(z) \\
    &\lesssim
    \frac{1-r}{d_K(rz,\widetilde E)^2}.
\end{aligned}
\label{eq:case1-quotient-derivative}
\end{equation}

Since $\beta_c-2=-1-d<0$, it follows from
\eqref{eq: Integral Norm Defn} and
\eqref{eq:case1-quotient-derivative} that
\begin{equation}
\left\|\frac{h}{h_r}\right\|_{\beta_c}^2
\lesssim
1+(1-r)^2
\int_0^1(1-\rho)^d
\int_{\mathbb T^2}
\frac{dm_2(\zeta)}
     {d_K(r\rho\zeta,\widetilde E)^4}\,d\rho,
\label{eq:case1-radial-integral}
\end{equation}
where $m_2$ denotes normalized Lebesgue measure on $\mathbb T^2$.

Write
$
    \zeta=(e^{i(u+v)},e^{i(u-v)}).
$
If
$
    \eta=(e^{i(s+q)},e^{i(s-q)})\in\widetilde E,
$
then
\[
    \frac{\langle r\rho\zeta,\eta\rangle}{2}
    =r\rho e^{i(u-s)}\cos(v-q).
\]
It follows that
\[
    d_K(r\rho\zeta,\widetilde E)
    \gtrsim
    (1-r)+(1-\rho)+\operatorname{dist}(e^{iu},E^1).
\]
Setting $b:=(1-r)+(1-\rho)$, and letting $m$ denote normalized
Lebesgue measure on $\mathbb T$, we obtain
\[
    \int_{\mathbb T^2}
    \frac{dm_2(\zeta)}
         {d_K(r\rho\zeta,\widetilde E)^4}
    \lesssim
    \int_{\mathbb T}
    \frac{dm(\xi)}
         {(b+\operatorname{dist}(\xi,E^1))^4}.
\]
Applying \cite[Exercise~9.4.1]{ElFallah2014_book} to
$\varphi(t)=(b+t)^{-4}$, and using
\eqref{eq:case1-neighborhood-growth}, we obtain
\[
\begin{aligned}
    \int_{\mathbb T}
    \frac{dm(\xi)}
         {(b+\operatorname{dist}(\xi,E^1))^4}
    &\lesssim
    1+\int_0^1
    \frac{|(E^1)_t|}{(b+t)^5}\,dt \\
    &\lesssim b^{-3-d}.
\end{aligned}
\]
Substituting this estimate into
\eqref{eq:case1-radial-integral}, and setting $x=1-\rho$, we obtain
\[
\begin{aligned}
    \left\|\frac{h}{h_r}\right\|_{\beta_c}^2
    &\lesssim
    1+(1-r)^2
    \int_0^1
    \frac{x^d}{(1-r+x)^{3+d}}\,dx \\
    &\lesssim1.
\end{aligned}
\]
Theorem~\ref{Radial_Dilation} shows that $h_{\widetilde E}^N$ is
cyclic in $\mathcal D_{\beta_c}(\mathbb D^2)$. Hence $h_{\widetilde E}$ is cyclic in
$\mathcal D_{\beta_c}(\mathbb D^2)$, and
Remark~\ref{rem:monotonicity} gives cyclicity in
$\mathcal D_\beta(\mathbb D^2)$ for every $\beta<\beta_c$.
\end{proof}

\subsection[Case 2]{Proof of Case 2:
\texorpdfstring{$1<\beta_c\leq2$}{1 < beta\_c <= 2}}
Set $d:=2-\beta_c\in[0,1)$. Choose a nonempty $\mathcal K$-set
$E\subset\Delta$ of Hausdorff dimension $d$. Write
\[
    E^1:=\pi_1(E)
    =\{\zeta\in\mathbb T:(\zeta,\zeta)\in E\}
\]
and
\[
    E^1_t
    :=\{\zeta\in\mathbb T:
          \operatorname{dist}(\zeta,E^1)\leq t\}.
\]
As in Case 1, we may choose $E$ so that
\begin{equation}
    |E^1_t|\lesssim t^{1-d}.
    \label{eq:case2-neighborhood-growth}
\end{equation}
Let $h_E$ be as in
Proposition~\ref{prop:ball-construction-E}.

\begin{prop}\label{prop:case2-noncyclicity}
The function $h_E$ is not cyclic in $\mathcal D_\beta(\mathbb D^2)$ whenever
$ \beta>\beta_c=2-d.$
\end{prop}

\begin{proof}
Assume first that $\beta_c<2$, and fix
$\beta_c<\beta<2.$
Since $d=2-\beta_c>2-\beta$, choose $s$ such that
$
    2-\beta<s<d.
$
By Frostman's lemma
\cite[Theorem~8.9]{Mattila1995}, there exists a probability measure $\mu$
supported on $E$ such that
\[
    \mu(B(\zeta,t))\lesssim t^s,
    \qquad \zeta\in E,\quad 0<t\leq1,
\]
where the balls are taken with respect to the Euclidean distance.

By the definition of the energy and
Lemma~\ref{lem:kernel-estimates}, we have
\[
\begin{aligned}
    I_\beta[\mu]
    &=
    \iint_{E\times E}
    k_\beta(\zeta,\eta)\,
    d\mu(\zeta)\,d\mu(\eta)
     \\
    &\leq
    \iint_{E\times E}
    |k_\beta(\zeta,\eta)|\,
    d\mu(\zeta)\,d\mu(\eta) \\
    &\lesssim
    \iint_{E\times E}
    \frac{d\mu(\zeta)\,d\mu(\eta)}
         {|\zeta-\eta|^{2-\beta}}.
\end{aligned}
\]
A straightforward application of Fubini's theorem and the layer-cake formula
\cite[p.~26]{LiebLoss} gives
\[
\begin{aligned}
    I_\beta[\mu]
    &\lesssim
    1+\int_E\int_0^1
    \frac{\mu(B(\zeta,t))}{t^{3-\beta}}\,
    dt\,d\mu(\zeta) \\
    &\lesssim
    1+\int_0^1t^{s+\beta-3}\,dt
    <\infty,
\end{aligned}
\]
because $s>2-\beta$. Thus
$\operatorname{cap}_\beta(E)>0,
$
and Theorem~\ref{thm:cap_noncyclic} shows that $h_E$ is not cyclic in
$\mathcal D_\beta(\mathbb D^2)$. Since $\beta\in(\beta_c,2)$ was arbitrary,
Remark~\ref{rem:monotonicity} gives the conclusion for every
$\beta>\beta_c$.

It remains to consider $\beta_c=2$. Let $\beta>2$. Choose $\zeta\in E$. By
Lemma~\ref{lem:bounded-point-evaluations}, point evaluation at
$\zeta$ is bounded on $\mathcal D_\beta(\mathbb D^2)$. Since $h_E(\zeta)=0$,
Lemma~\ref{lem:zero-at-bounded-point-evaluation} shows that $h_E$ is not
cyclic.
\end{proof}

The following lemma is important in the proof of our main result for $\beta_c\in(1,2].$
\begin{lem}\label{lem:diagonal-lifting}
For $u\in\mathcal O(\mathbb D)$, define
\[
    \mathcal Tu(z_1,z_2)
    :=\int_0^1u(tz_1+(1-t)z_2)\,dt.
\]
Then
$
    \|\mathcal Tu\|_\beta
    =\|u\|_{D_{\beta-1}(\mathbb D)}
$
for every $\beta\in\mathbb R$. Moreover,
$
    \|\mathcal Tu\|_{H^\infty(\mathbb D^2)}
    \leq\|u\|_{H^\infty(\mathbb D)}.
$
\end{lem}
\begin{rem} The critical index in the two-dimensional case is shifted by one unit compared to the one-dimensional case. This lifting lemma illustrates the reason behind this interesting phenomenon.
\end{rem}
\begin{proof}
Write $u(w)=\sum_{n\geq0}a_nw^n$. Since
\[
\begin{aligned}
    \int_0^1(tz_1+(1-t)z_2)^n\,dt
    &=\sum_{k=0}^n\binom nk z_1^kz_2^{n-k}
      \int_0^1t^k(1-t)^{n-k}\,dt \\
    &=\frac{1}{n+1}\sum_{k=0}^nz_1^kz_2^{n-k},
\end{aligned}
\]
we obtain
\[
    \mathcal Tu(z_1,z_2)
    =\sum_{n\geq0}\frac{a_n}{n+1}
      \sum_{k=0}^nz_1^kz_2^{n-k}.
\]
Consequently,
\[
\begin{aligned}
    \|\mathcal Tu\|_\beta^2
    &=\sum_{n\geq0}\sum_{k=0}^n
      \frac{|a_n|^2}{(n+1)^2}(n+1)^\beta \\
    &=\sum_{n\geq0}|a_n|^2(n+1)^{\beta-1}\\
    &=\|u\|_{D_{\beta-1}(\mathbb D)}^2.
\end{aligned}
\]
The $H^\infty$ estimate follows directly from the definition of
$\mathcal T$.
\end{proof}

The following lemma is inspired by \cite[Lemma~9.3.6]{ElFallah2014_book}.

\begin{lem}\label{lem:one-variable-deformation}
Let $\beta_c$, $d$, $E$, and $E^1$ be as above, and let $h_E$ be
the function given by Proposition~\ref{prop:ball-construction-E}. Set
$\widetilde h(w):=h_E(w,w)$ and $\rho_E(w):=\operatorname{dist}(w,E^1),$ for $w\in\mathbb{D}.$
Then $\widetilde h$ admits a holomorphic logarithm on $\mathbb D.$
Fix such a logarithm and define
\[
    \widetilde h^\lambda
    :=\exp\bigl(\lambda\log\widetilde h\bigr),
    \qquad \lambda\in\mathbb C.
\]
There exist $0<\theta<\pi/2$ and $c_\Omega>0$ such that, for
$\Omega:=\{\lambda\in\mathbb C\setminus\{0\}:|\arg\lambda|<\theta\},$
the following assertions hold:
\begin{enumerate}
    \item For every $\lambda\in\Omega$,
    $
        \widetilde h^\lambda\in D_{1-d}(\mathbb D)
    $
    and
    \[
        |\widetilde h^\lambda(w)|
        \lesssim_\lambda
        \rho_E(w)^{c_\Omega|\lambda|},
        \qquad w\in\mathbb D.
    \]

    \item The map
    $
        \lambda\longmapsto\widetilde h^\lambda
    $
    is holomorphic from $\Omega$ into $D_{1-d}(\mathbb D)$.

    \item We have
    $
        \lim_{\Omega\ni\lambda\to0}
        \|\widetilde h^\lambda-1\|_{D_{1-d}(\mathbb D)}
        =0.
    $
\end{enumerate}
\end{lem}

\begin{proof}
By \cite[proof of Theorem~4.3, p.~551]{Bruna1986}, the function used
in the construction of $h_E$ is the exponential of a holomorphic
function. Its restriction to the diagonal provides the required
logarithm of $\widetilde h$.

Since
$
    d_K\bigl((w,w),(\zeta,\zeta)\bigr)
    =|1-w\overline\zeta|,
$
the estimates for $h_E$ and $Rh_E$ in
Proposition~\ref{prop:ball-construction-E} give
\[
    |\widetilde h(w)|\asymp\rho_E(w),
    \qquad
    |\widetilde h'(w)|\lesssim1.
\]
Moreover, for $w \in \D$ and $\frac{1}{2}<r< 1$
\begin{equation}
    \rho_E(r w)
    \asymp
    (1-r)+\rho_E(w).
    \label{eq:rho-comparison}
\end{equation}
It follows that
\[
    \left|(\log\widetilde h)'(w)\right|
    =\left|\frac{\widetilde h'(w)}{\widetilde h(w)}\right|
    \lesssim\frac1{\rho_E(w)}.
\]
Therefore,
\[
    |\log\widetilde h(w)|
    \lesssim \left| \int_0^1 \frac{1}{\rho_E(rw) } \,dr \right| \asymp 1 + \left| \int_{\frac{1}{2}}^1 \frac{1}{(1-r)+\rho_E(w)} \,dr \right| \lesssim \log\frac{e}{\rho_E(w)}.
\]

Fix $C>0$ such that
\[
    |\operatorname{Im}\log\widetilde h(w)|
    \leq C\log\frac{e}{\rho_E(w)},
    \qquad w\in\mathbb D,
\]
and choose $\theta>0$ such that
$
    c_\Omega:=\cos\theta-C\sin\theta>0.
$
For $\lambda=|\lambda|e^{i\varphi}\in\Omega$,
$$
    \operatorname{Re}\lambda-C|\operatorname{Im}\lambda|
    \geq c_\Omega|\lambda|.
$$
Consequently,
\begin{equation}
    |\widetilde h^\lambda(w)|
    \lesssim_\lambda
    \rho_E(w)^{c_\Omega|\lambda|},
    \qquad
    |(\widetilde h^\lambda)'(w)|
    \lesssim_\lambda
    |\lambda|\rho_E(w)^{c_\Omega|\lambda|-1}.
    \label{eq:h-lambda-estimates}
\end{equation}
The implicit constants may be chosen uniformly for
$\lambda\in\Omega$ with $|\lambda|\leq1$.

We first prove assertion~1. Recall that
\begin{equation}
    \|u\|_{D_{1-d}(\mathbb D)}^2
    \asymp
    |u(0)|^2+
    \int_{\mathbb D}|u'(w)|^2(1-|w|)^d\,dA(w).
    \label{eq:one-variable-equivalent-norm}
\end{equation}
Set $q:=\min\{2c_\Omega|\lambda|,1\}$. By
\cite[Exercise~9.4.1]{ElFallah2014_book},
\eqref{eq:rho-comparison}, and
\eqref{eq:case2-neighborhood-growth}, we obtain, with $x=1-r$,
\begin{equation}
\begin{aligned}
    \int_{\mathbb D}
    \rho_E(w)^{q-2}(1-|w|)^d\,dA(w)
    &\lesssim
    1+\int_0^1\int_0^1
    \frac{x^dt^{1-d}}{(x+t)^{3-q}}\,dx\,dt\\
    &\lesssim
    1+\int_0^1t^{q-1}\,dt
    \lesssim\frac1q.
\end{aligned}
\label{eq:rho-power-integrability}
\end{equation}
Combining \eqref{eq:h-lambda-estimates},
\eqref{eq:one-variable-equivalent-norm}, and
\eqref{eq:rho-power-integrability}, we obtain
\[
\begin{aligned}
    \|\widetilde h^\lambda\|_{D_{1-d}(\mathbb D)}^2
    &\lesssim
    |\widetilde h^\lambda(0)|^2
    +|\lambda|^2
    \int_{\mathbb D}
    \rho_E(w)^{q-2}(1-|w|)^d\,dA(w)\\
    &\lesssim_\lambda
    1+\frac{|\lambda|^2}{q}
    <\infty.
\end{aligned}
\]

We next prove assertion~2. Fix $\lambda_0\in\Omega$ and a disk
$U\Subset\Omega$ centered at $\lambda_0$. The computation in
assertion~1 gives
$
    \sup_{\lambda\in U}
    \|\widetilde h^\lambda\|_{D_{1-d}(\mathbb D)}
    <\infty.
$

Choose $0<q<\min\{c_\Omega|\lambda_0|,1\}$. For $\lambda$ sufficiently
close to $\lambda_0$, \eqref{eq:h-lambda-estimates} gives
\[
\begin{aligned}
    |(\widetilde h^\lambda)'(w)
      -(\widetilde h^{\lambda_0})'(w)|^2
    &\lesssim
    \rho_E(w)^{2c_\Omega|\lambda|-2}
    +\rho_E(w)^{2c_\Omega|\lambda_0|-2}\\
    &\lesssim \rho_E(w)^{q-2}.
\end{aligned}
\]
By \eqref{eq:rho-power-integrability},
\[
    \int_{\mathbb D}
        \rho_E(w)^{q-2}(1-|w|)^d\,dA(w)<\infty.
\]
Moreover, for every $w\in\mathbb D$, we have
\begin{equation*}
    (\widetilde h^\lambda)'(w)
    \longrightarrow
    (\widetilde h^{\lambda_0})'(w)
    \qquad\text{and}\qquad
    \widetilde h^\lambda(0)
    \longrightarrow
    \widetilde h^{\lambda_0}(0).
\end{equation*}
Hence, dominated convergence and
\eqref{eq:one-variable-equivalent-norm} give
\[
    \|\widetilde h^\lambda-\widetilde h^{\lambda_0}\|
        _{D_{1-d}(\mathbb D)}
    \longrightarrow0.
\]
Since $\lambda\mapsto\widetilde h(w)^\lambda$ is holomorphic for
every $w\in\mathbb D$, the vector-valued Morera theorem
\cite[Chapter~3, Exercise~26(b), p.~89]{Rudin1991}
proves assertion~2.

Finally, for $\lambda$ sufficiently close to zero,
\eqref{eq:h-lambda-estimates},
\eqref{eq:one-variable-equivalent-norm}, and
\eqref{eq:rho-power-integrability} give
\[
\begin{aligned}
    \|\widetilde h^\lambda-1\|_{D_{1-d}(\mathbb D)}^2
    &\lesssim
    |\widetilde h^\lambda(0)-1|^2
    +\frac{|\lambda|^2}{c_\Omega|\lambda|}\\
    &=
    |\widetilde h^\lambda(0)-1|^2
    +\frac{|\lambda|}{c_\Omega}.
\end{aligned}
\]
Therefore,
$
    \lim_{\Omega\ni\lambda\to0}
    \|\widetilde h^\lambda-1\|_{D_{1-d}(\mathbb D)}
    =0,
$
which proves assertion~3.
\end{proof}

\begin{lem}\label{lem:lifted-analytic-deformation}
Let $\beta_c$, $d$, $E$, $\widetilde h$, $\rho_E$, $\Omega$, and
$c_\Omega$ be as in Lemma~\ref{lem:one-variable-deformation}, and let
$\mathcal T$ be the lifting operator from
Lemma~\ref{lem:diagonal-lifting}. For $\lambda\in\Omega$, define
$
    g_\lambda:=\mathcal T(\widetilde h^\lambda).
$
Then the following assertions hold.
\begin{enumerate}
    \item The map
    $
        \lambda\longmapsto g_\lambda
    $
    is holomorphic from $\Omega$ into
    $\mathcal D_{\beta_c}(\mathbb D^2)$ and
    $
        g_\lambda\longrightarrow1
        \quad\text{in }\mathcal D_{\beta_c}(\mathbb D^2)
    $
    as $\lambda\to0$ within $\Omega$.

    \item For every $\lambda\in\Omega$,
    \[
        |g_\lambda(z)|
        \lesssim_\lambda
        \operatorname{dist}(z,E)^{c_\Omega|\lambda|}
    \]
    for $z$ sufficiently close to $E$.

    \item There exists $L_\Omega>0$ such that
    $
        g_\lambda\in\mathcal M(\mathcal D_{\beta_c}(\mathbb D^2))
    $
    whenever $\lambda\in\Omega$ and $|\lambda|\geq L_\Omega$.
\end{enumerate}
\end{lem}

\begin{proof}
Since $\beta_c-1=1-d$, assertion~1 is an immediate consequence of
Lemmas~\ref{lem:diagonal-lifting} and
\ref{lem:one-variable-deformation}. In particular,
\[
    \|g_\lambda-1\|_{\beta_c}
    =
    \|\widetilde h^\lambda-1\|_{D_{1-d}(\mathbb D)}
    \longrightarrow0
\]
as $\lambda\to0$ within $\Omega$.

To prove assertion~2, let $\lambda\in\Omega$ and choose
$(\zeta,\zeta)\in E$ such that
$
    |z-(\zeta,\zeta)|
    \leq2\operatorname{dist}(z,E).
$
For $w_t:=tz_1+(1-t)z_2$, $0\leq t\leq1$, we have
\[
\begin{aligned}
    \rho_E(w_t)
    &\leq |w_t-\zeta|\\
    &\leq t|z_1-\zeta|+(1-t)|z_2-\zeta|\\
    &\lesssim\operatorname{dist}(z,E).
\end{aligned}
\]
Thus, by \eqref{eq:h-lambda-estimates},
\[
\begin{aligned}
    |g_\lambda(z)|
    &\leq
    \int_0^1|\widetilde h^\lambda(w_t)|\,dt\\
    &\lesssim_\lambda
    \operatorname{dist}(z,E)^{c_\Omega|\lambda|}.
\end{aligned}
\]

Finally, take $L_\Omega:=c_\Omega^{-1}$ and let
$\lambda\in\Omega$ satisfy $|\lambda|\geq L_\Omega$. Then \eqref{eq:h-lambda-estimates} gives
$
    \widetilde h^\lambda,
    (\widetilde h^\lambda)'
    \in H^\infty(\mathbb D).
$
Differentiating under the integral defining $g_\lambda$, we obtain
\[
    Rg_\lambda(z)
    =
    \int_0^1
    w_t(\widetilde h^\lambda)'(w_t)\,dt.
\]
Since $|w_t|\leq1$, it follows that
$
    g_\lambda,Rg_\lambda\in H^\infty(\mathbb D^2).
$
Proposition~\ref{prop:ball-restrictions-multipliers} now proves
assertion~3.
\end{proof}

\begin{prop}\label{prop:case2-cyclicity}
Let $d=2-\beta_c\in[0,1)$, let $E\subset\Delta$ be the
$\mathcal K$-set chosen above, and let $h_E$ be the function given by
Proposition~\ref{prop:ball-construction-E}. Then $h_E$ is cyclic in
$\mathcal D_\beta(\mathbb D^2)$ whenever
$
    \beta\leq\beta_c.
$
\end{prop}

\begin{proof}
It is enough to prove cyclicity in $\mathcal D_{\beta_c}(\mathbb D^2)$. Let
$g_\lambda$, $\Omega$, and $c_\Omega$ be as in
Lemma~\ref{lem:lifted-analytic-deformation}. By the standard comparison
between the Euclidean and Korányi distances,
$
    \operatorname{dist}(z,E)^2\lesssim d_K(z,E);
$
see \cite[p.~540]{Bruna1986}.

Let $\lambda\geq2c_\Omega^{-1}$ be real. By
Lemma~\ref{lem:lifted-analytic-deformation},
$
    g_\lambda\in\mathcal M(\mathcal D_{\beta_c}(\mathbb D^2))
$
and
\[
\begin{aligned}
    |g_\lambda(z)|
    &\lesssim_\lambda
    \operatorname{dist}(z,E)^{c_\Omega\lambda}\\
    &\lesssim
    \operatorname{dist}(z,E)^2
    \lesssim d_K(z,E)
    \asymp |h_E(z)|.
\end{aligned}
\]
Hence
$
    \frac{g_\lambda}{h_E}\in H^\infty(\mathbb D^2).
$
Since $g_\lambda$ and $h_E$ are multipliers,
Lemma~\ref{Perfekt} gives
$
    g_\lambda^2\in[h_E]$ for
    $ \lambda\geq2c_\Omega^{-1}.
$

As in Lemma~\ref{lem:lifted-analytic-deformation}, one can verify that $\lambda\longmapsto g_\lambda^2$ is holomorphic from $\Omega$ into $\mathcal D_{\beta_c}(\mathbb D^2)$. Let
$Q$ be the quotient map onto
$\mathcal D_{\beta_c}(\mathbb D^2)/[h_E]$. The holomorphic map
$
    \lambda\longmapsto Q(g_\lambda^2)
$
vanishes for every $\lambda$ such that $|\lambda|\geq2c_\Omega^{-1}$. Since
Banach-valued holomorphic functions are weakly holomorphic
\cite[Definition~3.30, p.~82]{Rudin1991}, the identity theorem
shows that this map vanishes identically on $\Omega$. Therefore,
\[
    g_\lambda^2\in[h_E],
    \qquad \lambda\in\Omega.
\]

Finally, arguing as in Lemma~\ref{lem:lifted-analytic-deformation}, we obtain
\[
    \lim_{\Omega\ni\lambda\to0}
    \|g_\lambda^2-1\|_{\beta_c}
    =0.
\]
Since $[h_E]$ is closed, we conclude that $1\in[h_E]$. Thus $h_E$ is cyclic in $\mathcal D_{\beta_c}(\mathbb D^2)$, and
Remark~\ref{rem:monotonicity} gives cyclicity in
$\mathcal D_\beta(\mathbb D^2)$ for every $\beta<\beta_c$.
\end{proof}

Propositions~\ref{prop:case1-noncyclicity},
\ref{prop:case1-cyclicity}, \ref{prop:case2-noncyclicity}, and
\ref{prop:case2-cyclicity} complete the proof of
Theorem~\ref{MainThm1}.

\section{Financial Support} The first named author was financially supported by the National
Science Center, Poland, under SHENG III research project 2023/48/Q/ST1/00048. Moreover, the first named author acknowledges financial support for his research stay at the University of South Florida from the program Excellence
Initiative at Jagiellonian University in Krak\'ow, Research Support Module 2026. The second named author also acknowledges the Institut de Mathématiques de Toulouse and the Department of Mathematics at Jagiellonian University for providing a stimulating research environment.

\section{Acknowledgments} We would like to thank Pascal Thomas, \L{}ukasz Kosi\'nski, Alan Sola, Dimitrios Vavitsas and Konstantinos Maronikolakis for reading our paper and providing us with detailed comments and suggestions that further enhanced the quality of the paper.

\section{AI disclosure}

Lemma~\ref{lem:diagonal-lifting}, in the form presented here, was suggested to us by ChatGPT. Although its proof is elementary, the lemma further simplified the proof of our main result in Case 2 ($\beta_c\in(1,2]$). We also used ChatGPT for proofreading and assistance with typesetting.

\bibliographystyle{plain} 

\end{document}